\documentclass{amsart}
\usepackage{amsfonts}

\usepackage{mathrsfs}
\usepackage{amscd}
\usepackage{amssymb}
\usepackage{amsfonts}
\usepackage{amsmath}

\theoremstyle{plain}
\newtheorem{theorem}{Theorem}[section]
\newtheorem{lemma}[theorem]{Lemma}

\newtheorem{proposition}[theorem]{Proposition}

\theoremstyle{definition}
\newtheorem{definition}[theorem]{Definition}

\theoremstyle{remark}

\begin{document}

\author{Xiaochen Zhou}
\address{University of Pennsylvania}
\email{zx@sas.upenn.edu}

\title{Recovery of the $C^{\infty}$ Jet from the Boundary Distance Function}

\begin{abstract}
For a compact Riemannian manifold with boundary, we want to find the
metric structure from knowledge of distances between boundary
points. This is called the ``boundary rigidity problem''. If the
boundary is not concave, which means locally not all shortest paths
lie entirely in the boundary, then we are able to find the Taylor
series of the metric tensor ($C^\infty$ jet) at the boundary (see
\cite{Semiglobal},\cite{Boundary}). In this paper we give a new
reconstruction procedure for the $C^{\infty}$ jet at non-concave
points on the boundary using the localized boundary distance
function.
\par A closely related problem is the ``lens rigidity problem'', which asks whether
the lens data determine metric structure uniquely. Lens data include
information of boundary distance function, the lengths of all
geodesics, and the locations and directions where geodesics hit the
boundary. We give the first examples that show that lens data cannot
uniquely determine the $C^{\infty}$ jet. The example include two
manifolds with the same boundary and the same lens data, but
different $C^\infty$ jets. With some additional careful work, we can
find examples with different $C^1$ jets, which means the boundaries
in the two lens-equivalent manifolds have different second
fundamental forms.
\end{abstract}

\maketitle

\section{Introduction}
Let $(M,g)$ be a compact Riemannian manifold with smooth boundary,
and let
\[\tau:M\times M\rightarrow\mathbb{R}\]
be the distance function given by $g$. The boundary rigidity problem
asks whether we can recover $g$ from the information of
$\tau|_{\partial M}$. That is, whether we can uniquely determine the
Riemannian metric of $M$, knowing the distances from boundary points
to boundary points. Obviously if we pullback the metric $g$ by a
diffeomorphism $f:M\rightarrow M$ that fixes every boundary point,
the resulting metric $f^\ast g$ gives the same boundary distance
function as before, but it is different from the original metric
$g$. So the natural question is, whether this is the only
obstruction to unique determination. If the answer is positive for
$(M,g)$, then it is called \textit{boundary rigid}.

\par One would like to know whether a given manifold with boundary
is boundary rigid. If in some cases we have affirmative answers, we
further want to have a procedure to recover the interior metric
structure from the information of boundary (``chordal'') distances.
The $C^\infty$ \textit{jet} at a point of a Riemannian manifold is,
roughly speaking, the Taylor series of the metric tensor at the
point. Therefore to recover the $C^\infty$ jet at boundary points is
the first step of the recovery of the entire interior metric
structure.

\par In arguments about the boundary rigidity problem, often one needs to extend
$(M,g)$ beyond its boundary, and here people care about the
smoothness of the extension. The extension of $g$ is smooth if and
only if the $C^\infty$ jets computed from both sides of the boundary
agree.

\par There are results on the boundary determination of $C^\infty$ jets.
Michel \cite{MichelC2} proved that boundary distances uniquely
determine the Taylor series of $g$ up to order 2, and in
\cite{MichelDim2} he proved the same result without order
limitations but with $\text{dim}(M)=2$, both with convex boundaries.
Here convexity roughly means that the distance of two sufficiently
close boundary points should be realized by a geodesic whose
interior does not intersect the boundary. In \cite{Semiglobal} there
is an elementary proof that the $C^\infty$ jet is uniquely
determined by the boundary distance function if the boundary is
convex. However the results above are not constructive. In
\cite{Boundary}, Uhlmann and Wang applied the result of
\cite{KeyIdentity} and used a suitably chosen reference metric, and
gave a recovery procedure of $C^\infty$ jet on the boundary from
localized boundary distance function. Here ``localized'' means we do
not need to know $\tau$ for all pairs of points in $\partial M$, but
we only should know $\tau$ restricted to an open neighborhood of the
diagonal of $\partial M\times\partial M$, that is, the distance
between close enough pairs. It should be noted that the arguments in
\cite{Semiglobal} and \cite{Boundary} also apply to non-concave
boundary (see Definition \ref{nonconcave}) without much
modification.

\par Up to now, the only result for concave boundary is
\cite{Local}, Theorem 1. The statement is, if a geodesic segment
$\gamma$ is tangent to the boundary at one end $p$, and the other
end $q$ lies on the boundary, then under some generic no conjugate
points condition, we can recover the $C^\infty$ jet at $p$ based on
the lengths of geodesic segments in a neighborhood of $\gamma$. The
argument in section \ref{recovery} of this paper is similar to
\cite{Local}.

\par In the first part (section \ref{pre} and \ref{recovery}) of this paper, we give the same results as in \cite{Boundary},
that is, a procedure to recover the $C^\infty$ jet at boundary
points, but our argument is relatively elementary. We also directly
adopt the weaker assumption that the boundary is non-concave, as
opposed to ``convex'' as in previous results.

\par In the second part (section \ref{examplec2} and \ref{examplec1}) of this paper, we give the first known
example that shows the lens data do not always determine the
boundary $C^\infty$ jet. Here ``lens data'' include the information
of $\tau|_{\partial M}$ and the lengths of all maximal geodesics,
together with the locations and the directions whenever they hit the
boundary (see Definition $\ref{lens}$, or \cite{Local} section 1 for
detail). So the results in \cite{Semiglobal}, \cite{Boundary}, and
the first part of this paper show that lens data uniquely determine
$C^\infty$ jet near non-concave points. Meanwhile, the results in
\cite{Local} should imply: We can uniquely recover $C^\infty$ jet
near ``generic'' concave points, from the lens data of geodesics
with bounded length, which are ``almost'' tangent to the boundary.
In the example in section \ref{examplec2}, the boundary is totally
geodesic, so nearby geodesics have unbounded length, although each
of them hits the boundary in finite time. In the example in section
\ref{examplec1}, the boundary is strictly concave, but every
complete geodesic tangent to the boundary has infinite length.
Therefore, the examples in this paper fall in the gap between
non-concave results (Theorem \ref{main} of this paper,
\cite{Semiglobal}, and \cite{Boundary}) and the concave result
\cite{Local}.
\par The author would like to thank Christopher Croke
for valuable advices and a careful reading of earlier drafts, and
Gunther Uhlmann for helpful comments.

\section{Preliminaries}\label{pre}
Throughout this paper we let $(M,g)$ be a compact $n$-dimensional
Riemannian manifold with smooth boundary $\partial M$. Let $\tau$ be
the distance function, and let $\rho=\tau^2$. We further introduce
the notation $\tau_x(\cdot)=\tau(\cdot,x)$, and
$\rho_x(\cdot)=\rho(\cdot,x)$. Notice that the distance might not be
realized by a geodesic, and a curve realizing it can have
non-geodesic parts in the boundary.

\par Let
\begin{equation}\label{mu}
\mu:\partial M\times\partial M\rightarrow\mathbb{R}
\end{equation}
be the distance function on the Riemannian manifold $(\partial
M,g|_{\partial M})$. Note that $\mu$ is not $\tau|_{\partial M}$
($\mu\geq\tau$ in general) although they may agree on some subset.
Near ``non-concave'' points, $\tau|_{\partial M}$ contains more
information than $\mu$.

\begin{definition}[Concave and Non-Concave points]\label{nonconcave}
 Let $x\in\partial M$. We say $\partial M$ is \textit{concave} at $x$ if the second fundamental form of $\partial M$ is negative
semi-definite at $x$, with respect to $\nu_x$ the inward-pointing
unit normal vector. We call $\partial M$ \textit{non-concave} at $x$
if it is not concave at $x$, that is, the second fundamental form
has at least one positive eigenvalue.
\end{definition}
In order to detect non-concave points from $\tau|_{\partial M}$, we
state the following elementary proposition without proof.
\begin{proposition}\label{Convex}
Let $x\in\partial M$. If $\partial M$ is concave in an open
neighborhood of $x$, then there exists $\varepsilon>0$ such that
whenever $p,q\in\partial M$ satisfy $\mu(x,p)<\varepsilon$ and
$\mu(x,q)<\varepsilon$, we have $\tau(p,q)=\mu(p,q)$. That means,
for a pair of points close enough to $x$, the shortest path between
them is along the boundary.
\end{proposition}

\par Since $(M,g)$ is extendable, we fix a compact $n$-dimensional Riemannian manifold without boundary
$(\tilde{M},\tilde{g})$, such that $\partial M$ is an
$(n-1)$-dimensional submanifold, the interior of $M$ is a connected
component of $\tilde{M}-\partial M$, and $g$ is the restriction of
$\tilde{g}$.

\par Next we define \textit{boundary normal coordinates} of $M$ near $\partial
M$. Let $(x_1,\ldots,x_{n-1})$ be a coordinate chart on the manifold
$\partial M$. For $p\in M$ close enough to $\partial M$, there is a
unique point $y=(y_1,\ldots,y_{n-1})\in\partial M$ that is closest
to $p$. We then give $p$ the coordinates $(y_1,\ldots,y_{n-1},y_n)$
where $y_n$ is defined to be the distance from $p$ to $y$. In such
coordinates, $g_{in}=\delta_{in}$, and the curves
$c(t)=(y_1,\ldots,y_{n-1},t)$ are geodesics perpendicular to
$\partial M$ at $t=0$.

\par Knowing the $C^\infty$ \textit{jet} of $g$ on $\partial M$ is equivalent to,
with respect to a given boundary normal coordinates, knowing the
derivatives
\[
\left.\frac{\partial^k}{\partial x_n^k} g_{ij}\right|_{\partial M}
\]
for all $k\geq 0$ and indices $i,j$, where
$g_{ij}=g\left(\frac{\partial}{\partial
x_i},\frac{\partial}{\partial x_j}\right)$. Clearly if we know the
jet with respect to one choice of boundary normal coordinates, we
are able to find the jet with respect to every choice of boundary
normal coordinates, knowing the coordinate change on the boundary.
For each integer $l\geq 0$, knowledge of $C^l$ jet means knowledge
of all the $\frac{\partial^k}{\partial x_n^k} g_{ij}$ with $k\leq
l$. In this paper we find the jet only under boundary normal
coordinates, and see \cite{Semiglobal} Theorem 2.1 for the precise
statement for general coordinates.

\par The key identity in the jet recovery procedure is the \textit{Eikonal
equation},
\[
|\nabla\tau_p|=1,\ \ p\in M,
\]
wherever the function $\tau_p$ is smooth. In coordinate charts the
Eikonal equation is
\[
g^{ij}\frac{\partial \tau_p}{\partial x_i}\frac{\partial
\tau_p}{\partial x_j}=1.
\]
Here we adopt Einstein summation convention, where $i,j$ ranges from
$1$ to $n$, and matrix $(g^{ij})$ is the inverse of $(g_{ij})$. In
boundary normal coordinates, this becomes
\begin{equation}
\label{Eikonal} g^{\alpha\beta}\frac{\partial \tau_p}{\partial
x_\alpha}\frac{\partial \tau_p}{\partial
x_\beta}+\left(\frac{\partial \tau_p}{\partial x_n}\right)^2=1,
\end{equation}
where $\alpha$ and $\beta$ range from $1$ to $(n-1)$.

\par We will use the convention that $i,j$ range from
$1$ to $n$, and $\alpha,\beta$ range from $1$ to $(n-1)$. We assume
we are always in boundary normal coordinates near $\partial M$. We
will write $\partial_{x_i}$ for $\frac{\partial}{\partial x_i}$, and
$\partial _{x_ix_j}$ for $\frac{\partial^2}{\partial x_i\partial
x_j}$, and so on. The reader should view the function $\tau$ as
$\tau(x,y)$ and
\[\tau(x_1,\ldots,x_n,y_1,\ldots,y_n),\]
so that formulas like $\partial_{x_i}\tau(p,q)$ and
$\partial_{y_i}\tau(p,q)$ will make sense. We treat $\rho=\tau^2$
similarly.

\section{Recovery of $C^\infty$ jet}\label{recovery}
We have the following lemmas.

\begin{lemma}\label{rho}
Let $V_\varepsilon\subset M\times M$ be the set of pairs $(x,y)$
satisfying the following properties: $\tau(x,y)$ is realized by a
geodesic in $M$, and $\tau(x,y)\leq\varepsilon$.
\par Then there exists an $\varepsilon>0$ such that $\rho$ is a smooth
function on $V_\varepsilon$.
\end{lemma}

The lemma is easy to prove if $M$ has no boundary. But when $M$ has
a boundary, we may first prove the property for the extension
$(\tilde{M},\tilde{g})$ with its corresponding $\tilde{\rho}$, and
use the fact that $\rho|_{V_\varepsilon}$ is the same as
$\tilde{\rho}|_{V_\varepsilon}$. Recall that a function being smooth
in (a subset of) a manifold with boundary (and possibly corner)
means the manifold together with the function can be extended into a
bigger one without boundary such that the function is still smooth.
\par Notice that we cannot replace $\rho$ in the last lemma with
$\tau$, because $\tau$ is not smooth where $x=y$. Smoothness is the
reason why we use distance squared rather than distance itself.

\begin{lemma}
\label{crho} Let $c:(-\varepsilon,+\varepsilon)\rightarrow M$ be a
smooth curve in $M$, which may intersect $\partial M$. If for each
$t$ the distance between $c(t)$ and $c(0)$ is realized by a
minimizing geodesic of $M$, then we have
\begin{equation}\label{crhoeq}
2|c'(0)|^2=\left.\frac{\partial^2}{\partial
t^2}\right|_{t=0}\rho(c(t),c(0)).
\end{equation}
\end{lemma}

\begin{proof}
If $c$ is a geodesic the statement is clearly true. If $c'(0)=0$ the
statement is also easy to prove.
\par Otherwise, we may think of
$c'(t)$ as coming from a vector field $X$ in a neighborhood of
$c(0)\in M$. This will give rise to a vector field $\tilde{X}=(X,0)$
in $M\times M$. Then we look at the right side,
\begin{eqnarray*}
\left.\frac{\partial^2}{\partial
t^2}\right|_{t=0}\rho(c(t),c(0))&=&\tilde{X}(\tilde{X}\rho)\\
&=&\text{Hess}\rho(\tilde{X},\tilde{X})+\left(\nabla_{\tilde{X}}\tilde{X}\right)\rho\\
&=&\text{Hess}\rho(\tilde{X},\tilde{X}),
\end{eqnarray*}
where all expressions are evaluated at $(c(0),c(0))$, a critical
point of $\rho$. However, $\text{Hess}\rho(\tilde{X},\tilde{X})$
only depends on $\tilde{X}$ at the point, which is $(c'(0),0)$, so
the right side of equation (\ref{crhoeq}) only depends on $c'(0)$.
This means we might as well assume $c$ is a geodesic.
\end{proof}

Now we are ready to recover the jet from the localized boundary
distance function. We present the recovery procedure in four steps,
$i.e.$ Proposition \ref{c0jet}, \ref{c1jet}, \ref{c2jet}, and
\ref{ckjet}.

\begin{proposition}\label{c0jet}
We can recover the $C^0$ jet from the localized boundary distance
function.
\end{proposition}
This is easy because $C^0$ jet is simply $g_{ij}|_{\partial M}$ the
Riemannian metric tensor. From the localized boundary distance
function, we are able to compute the length of any smooth curve in
$\partial M$. The curve lengths will tell us the metric tensor.

\par We start the recovery procedure for higher order jets. The
idea underlying the proofs of the following propositions
(Proposition \ref{c1jet}, \ref{c2jet}, and \ref{ckjet}) is derived
from \cite{Local}, section 3.

\begin{proposition}
\label{c1jet} If $\partial M$ is non-concave at $y$, then we can
recover the $C^1$ jet near $y\in\partial M$ from the localized
boundary distance function, with respect to a given boundary normal
coordinates.
\end{proposition}
We need a definition for the proof of this proposition.

\begin{definition} [Convex direction]
Let $\xi$ be a vector tangent to $\partial M$. We can find a
geodesic $\gamma:(-\varepsilon,+\varepsilon)\rightarrow\partial M$
with $\gamma'(0)=\xi$. (Here $\gamma$ may not be a geodesic in $M$.)
Let $\nabla$ be the covariant derivative in $M$, and $\nu$ the
inward-pointing unit normal at appropriate points in $\partial M$.
We call $\xi$ a \textit{convex direction} if
$\langle\nabla_{\gamma'(0)}\gamma',\nu\rangle>0$.
\end{definition}

Certainly the set of convex directions compose an open subset of
$T(\partial M)$. By definition, $\partial M$ is non-concave at $y$
if and only if there is at least one, and hence a nonempty open set
of convex directions based at $y$.

\begin{proof}[Proof of Proposition \ref{c1jet}]
After possibly changing coordinates, we assume $\partial_{x_1}$ is a
convex direction at $y$ (and in a neighborhood too). Let $c(t)$ be a
curve in $M$ such that $c'(t)=\partial_{x_1}$, which means its
coordinates representation is $(x_1+t,x_2,\ldots,x_n)$. Applying
lemma \ref{crho} we know

\begin{equation}
\label{g11} 2g_{11}(p)=\partial_{x_1x_1}\rho(p,p),
\end{equation}

where $p$ is not assumed to be on the boundary. Clearly both sides
of the equation are smooth functions of $p$. We now let the point
$p$ move in the direction $\partial_{x_n}$ and take the derivative
of equation (\ref{g11}),

\begin{eqnarray}
2\partial_{x_n}g_{11}&=&\partial_{x_n}\partial_{x_1x_1}\rho+\partial_{y_n}\partial_{x_1x_1}\rho \nonumber\\
&=&\partial_{x_1x_1}(\partial_{x_n}\rho+\partial_{y_n}\rho).\label{g11array}
\end{eqnarray}

We let $c:(-\varepsilon,+\varepsilon)\rightarrow\partial M$ be the
curve in $\partial M$ with $c(0)=y$ and $c'\equiv\partial_{x_1}$.
Since $\partial_{x_1}$ is a convex direction, we may assume for any
point $x$ on $c$ which is not the same point as $y$, the distance
between $y$ and $x$ is realized by a geodesic segment whose interior
does not intersect $\partial M$, and the geodesic is transversal to
$\partial M$ at both endpoints. So we know the value of
$(\partial_{x_n}\tau)(x,y)$ from first variation of arclength, and
similarly $(\partial_{y_n}\tau)(x,y)$. The values
$(\partial_{x_n}\rho)(x,y)$ and $(\partial_{y_n}\rho)(x,y)$ are then
easily recovered from localized $\tau|_{\partial M}$.

\par Since
$\partial_{x_1x_1}(\partial_{x_n}\rho+\partial_{y_n}\rho)|_{(y,y)}$
only depends on the value of
$(\partial_{x_n}\rho+\partial_{y_n}\rho)(x,y)$ where $x$ is along
the curve $c$, from equation (\ref{g11array}) we find
$\partial_{x_n}g_{11}|_y$.
\par Now we use the fact that a symmetric $n\times n$ tensor $(f_{ij})$ can
be recovered by knowledge of $f_{ij}v^i_kv^j_k$ for $N=n(n+1)/2$
``generic'' vectors $v_k$, $k=1,\ldots,N$, and we can find such $N$
vectors in any open set on the unit sphere.
\par We may choose appropriate $N$ perturbations of
$\partial_{x_1}$, say $v_k$, which are all convex directions at $y$.
Letting $(\partial_{x_n}g_{ij})$ be the tensor described above, We
find the values of $\partial_{x_n}g_{ij}v^i_kv^j_k$ using the same
method as above (change $\partial_{x_1}$ into $v_k$). They will tell
us the values of $\partial_{x_n}g_{ij}|_y$.
\end{proof}
Next we give the recovery procedure of $C^2$ jet, which applies
Eikonal equation. The cases of higher order jets are essentially the
same as $C^2$ jet.

\begin{proposition}\label{c2jet}
If $\partial M$ is non-concave at $y$, then we can recover the $C^2$
jet near $y\in\partial M$ from the localized boundary distance
function, with respect to a given boundary normal coordinates.
\end{proposition}

\begin{proof}
Clearly being non-concave is an open property for points in
$\partial M$, so we have already recovered $C^1$ jet near $y$ by
Proposition \ref{c1jet}. Again we assume without loss of generality
that $\partial_{x_1}$ is a convex direction at $y$.
\par Now look at equation (\ref{g11}) again, and let $p$ move
in the direction $\partial_{x_n}$, but this time we look at the
second derivative:
\begin{eqnarray}
2\partial_{x_nx_n}g_{11}&=&(\partial_{x_n}+\partial_{y_n})^2(\partial_{x_1x_1}\rho)\nonumber\\
&=&\partial_{x_1x_1}(\partial_{x_nx_n}\rho+2\partial_{x_ny_n}\rho+\partial_{y_ny_n}\rho).
\label{g22array}
\end{eqnarray}
Again we let $c$ be a short enough curve in $\partial M$ with
$c(0)=y$ and $c'\equiv\partial_{x_1}$. For the same reason as in the
proof of Theorem \ref{c1jet}, to compute the value of right side of
equation (\ref{g22array}) we only need to know
$(\partial_{x_nx_n}\rho+2\partial_{x_ny_n}\rho+\partial_{y_ny_n}\rho)$
at $(x,y)$ where $x$ lies on $c$.
\par If $x=y$, it is easy to see the value of
$(\partial_{x_nx_n}\rho+2\partial_{x_ny_n}\rho+\partial_{y_ny_n}\rho)$
at $(x,y)$ is $0$.
\par If $x\neq y$, then we look at the Eikonal equation as in (\ref{Eikonal}), in the
following form,
\begin{equation}\label{q1q2}
g^{\alpha\beta}(q_1)(\partial_{x_\alpha}\tau_{y}(x))(\partial_{x_\beta}\tau_{y}(x))+(\partial_{x_n}\tau_{y}(x))^2=1.
\end{equation}
Taking $\partial_{x_n}$ we get (with all terms evaluated at $x$)
\begin{equation}\label{q1q2long}
\partial_{x_n}g^{\alpha\beta}(\partial_{x_\alpha}\tau_{y})(\partial_{x_\beta}\tau_{y})+2g^{\alpha\beta}(\partial_{x_\alpha x_n}\tau_{y})(\partial_{x_\beta}\tau_{y})+2(\partial_{x_n}\tau_{y})(\partial_{x_nx_n}\tau_{y})=0.
\end{equation}
In equation (\ref{q1q2long}), the term
$\partial_{x_n}g^{\alpha\beta}$ we already know because
$(g^{\alpha\beta})$ is the inverse of $(g_{\alpha\beta})$ and we
know $g_{\alpha\beta}$ and $\partial_{x_n}g_{\alpha\beta}$. Also we
know $\partial_{x_\alpha}\tau_{q_2}$ from the localized boundary
distance function. We know $\partial_{x_\alpha x_n}\tau_{y}$ because
from the first variation formula we know $\partial_{x_n}\tau_{y}$ in
a neighborhood of $x$ along the boundary.
\par Therefore, so far the only term in equation (\ref{q1q2long}) we do not
know is $\partial_{x_nx_n}\tau_{y}(x)=\partial_{x_nx_n}\tau(x,y)$,
whose coefficient is $2\partial_{x_n}\tau_{y}(x)$, a nonzero number
because of the transversality of the segment between $x$ and $y$ to
$\partial M$. We can now immediately find value of
$\partial_{x_nx_n}\tau(x,y)$ from the other terms. Then, we can find
$\partial_{x_nx_n}\rho(x,y)$.
\par If we interchange the roles of $x$ and $y$, we can find
$\partial_{y_ny_n}\rho(x,y)$. As for $\partial_{x_ny_n}\rho(x,y)$,
we simply take derivative of equation (\ref{q1q2}) with respect to
$y_n$, that is, let $y$ move away from the boundary, and get
(assuming all are taken at $(x,y)$)
\[
2(\partial_{x_\alpha
y_n}\tau)(\partial_{x_\beta}\tau)+2\partial_{x_ny_n}\tau\partial_{x_n}\tau=0,
\]
where we know all but $\partial_{x_ny_n}\tau(x,y)$. So we can find
the value of $\partial_{x_ny_n}\tau(x,y)$ and hence
$\partial_{x_ny_n}\rho(x,y)$.

\par Up to now, we have computed
$(\partial_{x_nx_n}\rho+2\partial_{x_ny_n}\rho+\partial_{y_ny_n}\rho)$
at $(x,y)$ with $x\in c$, so by equation (\ref{g22array}), we can
find $\partial_{x_nx_n}g_{11}|_y$.
\par Once again, we perturb $\partial_{x_1}$ a little to get
sufficiently many vectors $v_k$ with convex directions. Carry out
the procedure for every $v_k$ to know
$(\partial_{x_nx_n}g_{ij})v^i_kv^j_k$, and combine the values all
together to find out all the $\partial_{x_nx_n}g_{ij}|_y$.
\end{proof}
We may now proceed by induction.
\begin{proposition}\label{ckjet}
Let $k\geq 3$. If we have recovered the $C^{k-1}$ jet in an open
neighborhood of $y\in\partial M$, then with respect to a given
boundary normal coordinates, we can recover the $C^k$ jet of the
same neighborhood from localized boundary distance function.
\end{proposition}
\begin{proof}
We let $p$ in equation (\ref{g11}) move towards the $n$th direction
and take the $k$th derivative, and get the following equation,
\begin{eqnarray}
2\partial_{x_n}^kg_{11}|_p&=&(\partial_{x_n}+\partial_{y_n})^k(\partial_{x_1x_1}\rho)|_{(p,p)}\\
&=&\partial_{x_1x_1}\left(\sum_{i=0}^k{k\choose
i}\partial_{x_n}^i\partial_{y_n}^{k-i}\rho\right)_{(p,p)}.
\label{gkkarray}
\end{eqnarray}
Here we borrow notation from Theorem \ref{c2jet}. The right side of
equation (\ref{gkkarray}) evaluated at $(y,y)$ only depends on
\begin{equation}\label{k-i}
\partial_{x_n}^i\partial_{y_n}^{k-i}\rho(x,y),
\end{equation}
where $x$ lies on the curve $c$, and $i=0,1,\ldots,k$.
\par If $x=y$ all are simple to compute, and the values do not even depend on the manifold.
\par If $x\neq y$, We first solve the problem when $i=k$. We apply the operator
$\partial_{x_n}^{k-1}$ to equation (\ref{q1q2}), recalling the
Eikonal equation holds wherever the gradient is smooth. The
resulting equation has terms involving $g^{\alpha\beta}$,
$\partial_{\alpha}\tau_{y}$, $\partial_{\beta}\tau_{y}$, and
$\tau_{q_2}$, and each of them may carry the operator
$\partial_{x_n}$ at most $(k-1)$ times, except the last term
\[
2(\partial_{x_n}\tau_{y})(\partial_{x_n}^k\tau_{y}),
\]
where $\partial_{x_n}\tau_{y}$ is nonzero at $x$ by transversality.
Since $x\neq y$ (which means $\tau\neq 0$), knowing the derivatives
of $\rho$ up to order $(k-1)$ is equivalent to knowing the
derivatives of $\tau$ up to order $(k-1)$. It is also okay to move
$x$ along the boundary ($i.e.$ take $\partial_{x_\alpha}$) because
all procedures work in some open neighborhoods, with a change of
coordinates if necessary. So by the inductive hypothesis, we know
$g^{\alpha\beta}$, $\partial_{x_\alpha}\tau$,
$\partial_{x_\beta}\tau$, and $\tau$, along with their derivatives
involving $\partial_{x_n}$ up to $(k-1)$ times. Therefore, we can
compute $\partial_{x_n}^k\tau$, and hence $\partial_{x_n}^k\rho$.
Now we have finished the case $i=k$.
\par If $i=0$, we do the same procedure after interchanging
$x$ and $y$.
\par Finally, if $0<i<k$, we have at least one $\partial_{x_n}$ and
one $\partial_{y_n}$ applied to $\rho$ in formula (\ref{k-i}). To
proceed, we can apply $\partial_{x_n}^{i-1}\partial_{y_n}^{k-i}$ to
Eikonal equation $(\ref{q1q2})$. We then use the same method as in
the case $i=k$. Note that $\partial_{y_n}g^{\alpha\beta}(x)\equiv0$,
because $\partial_{y_n}$ does not move point $x$.
\par So far we have found $\partial_{x_n}^kg_{11}|_y$.
\par We perturb $\partial_{x_1}$ a little to get sufficiently many
vectors $v$ with convex directions. Carry out the procedure for
every such $v$ to know $(\partial_{x_n}^kg_{ij})v^iv^j$, and put the
results all together to determine all the
$\partial_{x_n}^kg_{ij}|_y$.
\end{proof}

If we combine the results of Proposition $\ref{c1jet}$, Proposition
$\ref{c2jet}$, and Proposition $\ref{ckjet}$, we have the following
\begin{theorem}\label{main}
Suppose $\partial M$ is non-concave at $y$, and $D\subset\partial
M\times\partial M$ is an open neighborhood of $(y,y)$. Then we can
recover the $C^\infty$ jet of $g$ at $y$ based on the information of
$\tau|_D$.
\end{theorem}
If we want to weaken the assumption in the theorem, we can try to
detect non-concave points of $\partial M$ by information about
$\tau|_{\partial M}$ only. The contrapositive statement of
Proposition \ref{Convex} is, if in any open neighborhood of $y$ in
$\partial M$, we can find $x_1,x_2$ with
$\tau(x_1,x_2)<\mu(x_1,x_2)$, then $y$ is not in the interior of the
(closed) set of concave points, $i.e.,$ $y$ is in the closure of
non-concave points. But we can recover $C^{\infty}$ jets near
non-concave points, and jets are continuous (because $g$ is
extendable), so we know the jet at $y$.
\begin{theorem}\label{main2}
Suppose $y\in\partial M$. If for every neighborhood $D$ of $(y,y)\in
\partial M\times\partial M$, we have $\tau|_D$ and $\mu|_D$ do not entirely agree,
then we can recover $C^\infty$ jet of $g$ at $y$.
\end{theorem}
This can help us know the interior metric structure if we a priori
assume the manifold, metric, and boundary are analytic. Observe that
the set of non-concave points is open, and we have the following
\begin{theorem}\label{analytic}
Suppose $(M,\partial M,g)$ is analytic. If for any connected
component of $\partial M$, we have a point $y$ satisfying the
hypothesis of Theorem \ref{main}, then we can recover the $C^\infty$
jet of $g$ at all points of $\partial M$.
\end{theorem}
This can lead to lens rigidity results in the category of analytic
metrics, with some assumptions such as ``every unit speed geodesic
hits the boundary in finite time'', see \cite{VargoAnalytic}.

\par In Theorem \ref{main2} and \ref{analytic}, the hypothesis is simply ``the localized chordal distance function
at the boundary does not agree with the localized in-boundary
distance function''. One is tempted to remove the words
``localized'', which means we now have the question: for an analytic
Riemannian manifold with boundary, if $\tau$ does not entirely agree
with $\mu$, can we compute the $C^\infty$ jet? The answer is
negative, because of the examples described in the next section.

\section{Examples of different $C^2$ jets}\label{examplec2}

In this section we are going to give an example of two manifolds
which have the same boundary and the same lens data but different
$C^\infty$ jets. The idea of the example is borrowed from
\cite{Rigidity} section 2, and \cite{Conjugacy} section 6. The idea
in \cite{Rigidity} and \cite{Conjugacy} is, if we have a surface of
revolution with two circles as boundary, then in some sense, the
lens data only depends on the measures of the sublevel sets of
radius function along a meridian. We can find distinct smooth
functions $f_1,f_2$ both with domain $[a,b]$, such that they have
the same measure for every sublevel set.

\par Before giving the example, we give the definition of lens data
and lens equivalence.

\begin{definition}\label{lens}
Let $(M,\partial M,g)$ be a Riemannian manifold with boundary, and
let $\partial(SM)$ be the set of unit vectors with base point at
boundary. Define set $\Omega\subset
\partial(SM)\times\partial(SM)\times\mathbb{R}^+$ to be the set of
3-tuples $(\gamma'(0),\gamma'(T),T)$ that satisfies: (1) $\gamma$ is
a unit speed geodesic, (2) $\langle\gamma'(0),\nu\rangle>0$ $i.e.$
$\gamma'(0)$ points inwards, and (3) $T$ is the first moment at
which $\gamma$ hits $\partial M$ again. The description above
depends on the interior structure, so we orthogonally project
$\partial(SM)$ to $\overline{B(\partial M)}$ the closed ball bundle
on $\partial M$. This projection maps $\Omega$ to $\Omega'\subset
B(\partial M)\times\overline{B(\partial M)}\times\mathbb{R}^+$.
\par We define \textit{lens data} to be the information of $\Omega'$ and
$\tau|_{\partial M}$. We say two Riemannian manifolds with boundary
are \textit{lens equivalent} if they have the same boundary and the
same lens data $i.e.$ the same $\Omega'$ and $\tau|_{\partial M}$.

\end{definition}

\par Consider the strip $S$ defined as $\mathbb{R}\times[0,L]$,
with standard coordinates $(x,y)$ where $0\leq y\leq L$. Obviously
$S$ has a natural structure of manifold with boundary. Define a
Riemannian metric $g$ on $S$ as
\[
g_{yy}=1,\ \ \ \ g_{xy}=0,\ \ \ \ g_{xx}=f(y).
\]
Here $f:[0,L]\rightarrow\mathbb{R}$ is a smooth function, such that
$f(0)=f(L)=1$ and $f(y)\geq 1$ for all $y\in(0,L)$. Under certain
circumstances, this manifold can be viewed as the universal cover of
a surface of revolution in $\mathbb{R}^3$.
\par Obviously the curves $\gamma_0(t)=(x_0,t)$ and
$\gamma_L(t)=(x_0,L-t)$ are unit speed geodesics, for any
$x_0\in\mathbb{R}$. So the normal ($i.e.$ ``$n$-th'' in previous
sections) direction is simply the $y$ direction, and the $C^\infty$
jet of $S$ is determined by
\[
\frac{\partial^k f}{\partial y^k},\ k=1,2,\ldots.
\]
If we let index $1$ stand for $x$ and $2$ for $y$, a straightforward
computation of Christoffel symbols shows
\[
\Gamma_{22}^2=\Gamma_{12}^2=\Gamma_{22}^1=\Gamma_{11}^1=0,\ \ \ \
\Gamma_{11}^2=-\frac{1}{2}f'(y),\ \ \ \
\Gamma_{12}^1=\frac{1}{2}\frac{f'(y)}{f(y)}.
\]
\par Let $(x(t),y(t))$ be a geodesic of $S$ parametrized by arclength. Then it will satisfy the second order system
\[
\frac{d^2x_k}{dt^2}+\sum_{i,j}\Gamma_{ij}^k\frac{dx_i}{dt}\frac{dx_j}{dt},\
\ k=1,2,
\]
where $i,j$ ranges over $1,2$, and $x_1=x$, $x_2=y$.
\begin{lemma} Along a geodesic, $\frac{dx}{dt}\cdot f(y)$ is
constant.\end{lemma}
\begin{proof}
\begin{eqnarray*}
\frac{d}{dt}\left(\frac{dx}{dt}\cdot f(y)\right)&=&\frac{dx}{dt}\cdot\frac{d}{dt}f(y)+\frac{d^2x}{dt^2}\cdot f(y)\\
&=&\frac{dx}{dt}\frac{dy}{dt}f'(y)-f(y)\cdot\sum_{i,j}\Gamma_{ij}^1\frac{dx_i}{dt}\frac{dx_j}{dt}\\
&=&\frac{dx}{dt}\frac{dy}{dt}f'(y)-f(y)\cdot2\Gamma_{12}^1\frac{dx}{dt}\frac{dy}{dt}\\
&=&\frac{dx}{dt}\frac{dy}{dt}f'(y)-f(y)\cdot\frac{f'(y)}{f(y)}\frac{dx}{dt}\frac{dy}{dt}\\
&=&0.
\end{eqnarray*}
\end{proof}

The Lemma above is Clairaut's relation when $S$ is a surface of
revolution. The Lemma does not require that the geodesic is unit
speed, but from now on we assume all geodesics in discussion are of
unit speed.
\par Since $f(y)$ is never 0, we know either $\frac{dx}{dt}$
is constant zero or never changes sign. Since $g_{ij}=\delta_{ij}$
at the boundary, we know each geodesic leaves $S$ at the same angle
as when it enters $S$. Also, in each geodesic,
$\left|\frac{dx}{dt}\right|$ assumes its maximum on the boundary
because $f(y)$ is minimal there. Therefore, since
\[
\left(\frac{dy}{dt}\right)^2+\left(\frac{dx}{dt}\right)^2g_{xx}=1,
\]
we know $\frac{dy}{dt}$ never changes sign in the interior. This
means each entering geodesic transversal to the boundary goes all
the way to the other component of the boundary, and hits the
boundary with the ``same'' direction as it entered.

\par Suppose $(x(t),y(t))$, $t\in[0,T]$ is such a maximal geodesic,
and without loss of generality we assume $\frac{dy}{dt}>0$ which is
equivalent to the geodesic entering $S$ at $y=0$ and leaving $S$ at
$y=L$. Since $\frac{dy}{dt}$ is positive and smooth, we have
\begin{eqnarray*}
T&=&\int_0^L\frac{dt}{dy}dy\\
&=&\int_0^L\left(\frac{dy}{dt}\right)^{-1}dy\\
&=&\int_0^L\left(1-x'(t)^2\cdot f(y)\right)^{-\frac{1}{2}}dy\\
&=&\int_0^L\left(1-\frac{x'(0)^2}{f(y)}\right)^{-\frac{1}{2}}dy,
\end{eqnarray*}
and
\begin{eqnarray*}
x(T)-x(0)&=&\int_0^L\frac{dx}{dy}dy\\
&=&\int_0^L\frac{dx}{dt}\left(\frac{dy}{dt}\right)^{-1}dy\\
&=&\int_0^L\frac{x'(0)}{f(y)}\left(1-\frac{x'(0)^2}{f(y)}\right)^{-\frac{1}{2}}dy,
\end{eqnarray*}

Now let's consider two different strips of this kind, $S_1,S_2$ with
$L=2\pi$ and
\begin{eqnarray*}
f_1(y)&=&2-\cos(y),\\
f_2(y)&=&2-\cos(2y).
\end{eqnarray*}
Consider a geodesic in $S_1$ and one in $S_2$ entering them at the
same location and same direction, $i.e.$, $x_1(0)=x_2(0)$ and
$x_1'(0)=x_2'(0)$. Then obviously $T_1=T_2$ and $x_1(T_1)=x_2(T_2)$,
because for each real number $r$, the sublevel sets $\{f_1\leq r\}$
and $\{f_2\leq r\}$ have the same measure.
\par If we take quotients of $S_1$ and $S_2$, both by $x$-axis slides of multiples of 100,
we have two cylinders with identical lens data but different
$C^\infty$ jets. Furthermore, both are compact and analytic. If we
want the boundary to be connected, we can take the quotients of the
cylinders by an orientation-reversing involution, which gives us two
M\"{o}bius bands.

\begin{theorem}
There is an example of two analytic Riemannian manifolds with
isometric boundaries and identical lens data, but different
$C^\infty$ jets at the boundaries.
\end{theorem}

The examples $S_1$ and $S_2$ have same lens data, same $C^1$ jet,
but different $C^2$ jet. If one wants a pair of examples of
different $C^1$ jets, then the idea still works, but to construct an
example we need to care about the smoothness at the peaks and the
smooth extendability at boundary, as in the following section.

\section{Examples of different $C^1$ jets}\label{examplec1}

In this section we give an example of two manifolds which have the
same boundary and the same lens data but different $C^1$ jets.
Knowledge of the $C^1$ jet is equivalent to knowledge of the second
fundamental form of the boundary as a submanifold, so different
$C^1$ jet means different ``shape'' of the embedding.
\par The setup is the same as
in the previous section. The only modifications are the functions
$f_1$ and $f_2$. Let $L=14$.
\par Let $f_1:[0,14]\rightarrow\mathbb{R}$ be a smooth function that satisfies the following properties:
\begin{eqnarray*}
&&f_1(x)=x+1,\ \text{if }0\leq x\leq 1;\\
&&f_1'(x)>0,\ \text{if }x\in[1,3),\ \text{and }f_1'(3)=0;\\
&&f_1(3+t)=f_1(3-t)\ \text{if }t\in[0,1];\\
&&f_1'(x)<0,\ \text{if }x\in(3,6);\\
&&f_1(x)=1,\ \text{if }6\leq x\leq 7;\\
&&f_1(7+t)=f_1(7-t)\ \text{if }t\in[0,7].
\end{eqnarray*}
We have some freedom of choice here, but it is crucial and possible
to make $f_1$ smooth. Intuitively, $f_1$ starts at $1$ and increases
linearly in the first time period, near the peak $f_1$ is symmetric,
then it smoothly decreases to the constant function $1$, and later
it copies its own mirror image.
\par In order to define $f_2$, we define the non-increasing function $H:[1,+\infty)\rightarrow[0,6]$,
\[
H(y)=m(\{x\in[0,6]\ |\ f_1(x)\geq y\}),
\]
where $m(\cdot)$ is the Lebesgue measure.
\par We think of $f_2$ as the ``horizontal central lineup'' of $f_1$. That is, $f_2:[0,14]\rightarrow L$ should satisfy the following:
\begin{eqnarray*}
&&\text{if }x\in[0,3),\ \text{then } f_2(x)=y\ \text{if and only if }x=3-\frac{H(y)}{2};\\
&&f_2(3)=f_1(3);\\
&&f_2(3+t)=f_2(3-t),\ \text{if }t\in[0,3];\\
&&f_2(x)=1,\ \text{if }x\in[6,7];\\
&&f_2(7+t)=f_2(7-t)\ \text{if }t\in[0,7].
\end{eqnarray*}
Obviously $f_2$ is uniquely determined by $f_1$, and they have the
same measure for every sublevel set. The only possibilities of
non-smoothness of $f_2$ are at $0,3,6,8,11,14$. We have $f_2=f_1$
near $3$ and $11$ from the symmetry of $f_1$ near the peaks. It is
not hard to see $f_2$ is smooth at $6$, because the graph of $f_2$
near $(6,1)$ is a linear transformation of that of $f_1$ near the
same point. The smoothness near $8$ is guaranteed for the same
reason. Finally, from the symmetry of $f_2$, the smooth
extendability of $f_2$ at $0$ and $14$ immediately follows.

\par Observe that $f_1'(0)=1=-f_1'(14)$ but $f_2'(0)=f_2'(14)=0$, $i.e.$, the boundary is concave in $S_1$ but totally geodesic in $S_2$.
We now use the same argument as in last section. The strips
$S_1,S_2$ defined by $f_1,f_2$ are lens equivalent, but have
different $C^1$ jets. If we want, we can take quotients to make the
strips compact, and to make the boundaries connected, as in the last
section.

\begin{theorem}
There is an example of two Riemannian manifolds with isometric
boundaries and identical lens data, but different $C^1$ jets at the
boundaries.
\end{theorem}


\begin{thebibliography}{5}

\bibitem{Rigidity}
C. Croke, \emph{Rigidity and the distance between boundary points},
J. Differential Geom. 33 (1991) 445-464.

\bibitem{Conjugacy}
C. Croke \& B. Kleiner, \emph{Conjugacy and rigidity for manifolds
with a parallel vector field}, J. Differential Geom. 39 (1994)
659-680.

\bibitem{Semiglobal}
M. Lassas, V. Sharafutdinov \& G. Uhlmann, \emph{Semiglobal boundary
rigidity for Riemannian metrics}, Math. Ann. 325 (2003), 767-793.

\bibitem{MichelC2}
R. Michel, \emph{Sur la rigidit\'{e} impos\'{e}e par la longuer des
g\'{e}od\'{e}siques}, Invent. Math. 65 (1981), 71-83.

\bibitem{MichelDim2}
R. Michel, \emph{Restriction de la distance g\'{e}od\'{e}sique \`{a}
un arc et rigidit\'{e}}, Bull. Soc. Math. France 122 (1994),
435-442.

\bibitem{KeyIdentity}
P. Stefanov \& G. Uhlmann, \emph{Rigidity for metrics with the same
lengths of geodesics}, Math. Res. Lett. 5 (1998) 83-96.

\bibitem{Local}
P. Stefanov \& G. Uhlmann, \emph{Local lens rigidity with incomplete
data for a class of non-simple Riemannian manifolds}, J.
Differential Geom. 82 (2009) 383-409.

\bibitem{Boundary}
G. Uhlmann \& J. Wang, \emph{Boundary determination of a Riemannian
metric by the localized boundary distance function}, Adv. in Appl.
Math. 31 (2003) 379-387.

\bibitem{VargoAnalytic}
J. Vargo, \emph{A proof of lens rigidity in the category of analytic
metrics}, Math. Res. Lett. 16 (2009), no. 6, 1057-1069
\end{thebibliography}
\end{document}